\documentclass[11pt]{amsart}
\usepackage{verbatim}
\usepackage{rotating} 
\usepackage{amssymb}
\usepackage[bookmarks,colorlinks,breaklinks]{hyperref}  
\hypersetup{linkcolor=blue,citecolor=blue,filecolor=blue,urlcolor=blue}
\usepackage{mathrsfs,amsfonts,amsmath,amssymb,epsfig,amscd,xy,amsthm} 

\newtheorem{theorem}{Theorem}[section]
\newtheorem{proposition}[theorem]{Proposition}

\newtheorem{corollary}[theorem]{Corollary}

\newtheorem{definition}[theorem]{Definition}

\newtheorem{conjecture}[theorem]{Conjecture}

\theoremstyle{plain}

\theoremstyle{remark}
\newtheorem{claim}[theorem]{Claim}
\newtheorem{remark}[theorem]{Remark}
\newtheorem{example}[theorem]{Example}

\newcommand{\C}{{\mathbb C}}
\newcommand{\F}{{\mathbb F}}
\newcommand{\Q}{{\mathbb Q}}

\newcommand{\Z}{{\mathbb Z}}
\newcommand{\N}{{\mathbb N}}

\newcommand{\Qbar}{\bar{\Q}}

\newcommand{\bG}{{\mathbb G}}

\newcommand{\bF}{{\mathbb F}}
\newcommand{\Fq}{\bF_q}
\newcommand{\lra}{\longrightarrow}
\newcommand{\cO}{\mathcal{O}}

\author{Dragos Ghioca}
\address{
Dragos Ghioca\\
Department of Mathematics\\
University of British Columbia\\
Vancouver, BC V6T 1Z2\\
Canada
}
\email{dghioca@math.ubc.ca}


\begin{document}
	\title[The dynamical Mordell-Lang problem]{The dynamical Mordell-Lang conjecture in positive characteristic}

\begin{abstract}
Let $K$ be an algebraically closed field of prime characteristic $p$, let $N\in\N$, let $\Phi:\bG_m^N\lra \bG_m^N$ be a self-map defined over $K$, let $V\subset \bG_m^N$ be a curve defined over $K$, and let $\alpha\in\bG_m^N(K)$. We show that the set $S=\{n\in\N\colon \Phi^n(\alpha)\in V\}$ is a union of finitely many arithmetic progressions, along with a finite set and finitely many $p$-arithmetic sequences, which are sets of the form $\{a+bp^{kn}\colon n\in\N\}$ for some $a,b\in\Q$ and some $k\in\N$. We also prove that our result is sharp in the sense that $S$ may be infinite without containing an arithmetic progression. Our result addresses a positive characteristic version of the dynamical Mordell-Lang conjecture and it is the first known instance when a structure theorem is proven for the set $S$ which includes $p$-arithmetic sequences.
	\end{abstract}

	\maketitle


 \section{Introduction}
\label{sec:introduction}

In this paper, as a matter of convention, any subvariety of a given variety is assumed to be closed.  
We denote by $\N$ the set of positive integers, we let 
$\N_0:=\N\cup\{0\}$, and let $p$ be a prime number. 
An arithmetic progression is a set of the form
$\{mk+\ell:\ k\in\N_0\}$  
for some $m,\ell\in\N_0$; note that when $m=0$, this set
is a singleton. 
For a set $X$ endowed with a self-map $\Phi$, and for $m\in\N_0$, we let $\Phi^m$
denote the $m$-th iterate $\Phi\circ\cdots\circ \Phi$, where $\Phi^0$
denotes the identity map on $X$. If $\alpha\in X$, we define the orbit
$\cO_\Phi(\alpha):=\{\Phi^n(\alpha)\colon n\in\N_0\}$.  

Motivated by the classical Mordell-Lang conjecture proved by Faltings \cite{Fal94} (for abelian varieties) and Vojta \cite{Voj96} (for semiabelian varieties), the dynamical Mordell-Lang Conjecture predicts that for a quasiprojective variety $X$ endowed with a self-map $\Phi$ defined over a field $K$ of characteristic $0$, given a point 
$\alpha\in X(K)$ and a subvariety $V$ of $X$, the set $$S:=\{n\in\N\colon \Phi^n(\alpha)\in V(K)\}$$ 
is a finite union of arithmetic progressions (see \cite[Conjecture~1.7]{GT-JNT} along with the earlier work  of Denis \cite{Den94} and Bell \cite{Bel06}). Considering $X$ a semiabelian variety and $\Phi$ the translation by a point $x\in X(K)$, one recovers the cyclic case in the classical Mordell-Lang conjecture from the above stated dynamical Mordell-Lang Conjecture; we refer the readers to \cite{book} for a survey of recent work on the dynamical Mordell-Lang conjecture. 

With the above notation for $X,\Phi,K,V,\alpha,S$, if $K$ has characteristic $p>0$ then $S$ may be infinite without containing an infinite arithmetic progression (see \cite[Example~3.4.5.1]{book} or our Example~\ref{a b non-integers}). Similar to the classical Mordell-Lang conjecture in characteristic $p$, one has to take into account varieties defined over finite fields. So, motivated by the structure theorem of Moosa-Scanlon describing in terms of $F$-sets the intersection of a subvariety of a semiabelian variety (in positive characteristic) with a finitely generated group (see \cite{Moosa-Scanlon} and also \cite{groups}), the following conjecture was proposed in \cite[Conjecture~13.2.0.1]{book}.   

\begin{conjecture}[(Ghioca-Scanlon) Dynamical Mordell-Lang Conjecture in positive characteristic]
\label{char p DML}
Let $X$ be a quasiprojective variety defined over a field $K$ of characteristic $p$. Let $\alpha\in X(K)$, let $V\subseteq X$ be a subvariety defined over $K$, and let $$\Phi:X\lra X$$ be an endomorphism defined over $K$. Then the set $S:=S(X,\Phi,V,\alpha)$ of integers $n\in\N_0$ such that $\Phi^n(\alpha)\in V(K)$ is a union of finitely many arithmetic progressions along with finitely many sets of the form
\begin{equation}
\label{form of the solutions}
\left\{\sum_{j=1}^{m} c_j p^{k_j n_j}\text{ : }n_j\in\N_0\text{ for each }j=1,\dots m\right\},
\end{equation}
for some $c_j\in\Q$, and some $k_j\in\N_0$. 
\end{conjecture}

Clearly, the set $S$ may contain infinite arithmetic progressions if $V$ contains some periodic subvariety under the action of $\Phi$ which intersects $\cO_\Phi(\alpha)$. On the other hand, it is possible for $S$ to contain nontrivial sets of the form \eqref{form of the solutions} (where $m$ is even larger than $1$; see Example~\ref{2 powers of p}). 

\begin{example}
\label{2 powers of p}
Let $p>2$, let $K=\F_p(t)$, let $X=\bG_m^3$, let $\Phi:\bG_m^3\lra \bG_m^3$ given by $\Phi(x,y,z)=(tx,(1+t)y, (1-t)z)$, let $V\subset \bG_m^3$ be the hyperplane given by the equation $y+z-2x=2$, and let $\alpha=(1,1,1)$. Then one easily checks that the set $S$ from Conjecture~\ref{char p DML} consists of all numbers of the form $p^{n_1}+p^{n_2}$ for $n_1,n_2\in\N_0$.
\end{example}

In \cite[Theorem~1.4]{noetherian} (see also \cite{Clay}), it is proven that the set $S$ from Conjecture~\ref{char p DML} is a union of finitely many arithmetic progressions along with a set of Banach density $0$. However, 
Conjecture~\ref{char p DML} predicts a much more precise description of the set $S$. The only nontrivial case when Conjecture~\ref{char p DML} was known to hold prior to our present paper is the case of group endomorphisms $\Phi$ of algebraic tori $X=\bG_m^n$ under the additional assumption that the action of $\Phi$ on the tangent space at the identity of $X$ is given by a diagonalizable matrix (see \cite[Proposition~13.3.0.2]{book}). Under these assumptions, one proves that the set $S$ consists only of finitely many arithmetic progressions, i.e., the more complicated sets of the form \eqref{form of the solutions} \emph{do not appear}. 
In this paper we prove the first important partial result towards Conjecture~\ref{char p DML} by showing it holds for any curve $V$ contained in $\bG_m^N$ endowed with an arbitrary self-map $\Phi$ (not necessarily a group endomorphism). Since our result has no restriction on $\Phi$, our proof is more complicated than the proof of \cite[Proposition~13.3.0.2]{book} (knowing that $\Phi$ is a group endomorphism with its Jacobian at the identity being a diagonalizable matrix simplifies significantly the arguments since the sets of the form \eqref{form of the solutions} do not appear in the conclusion). In our Theorem~\ref{main result}, we encounter sets of the form \eqref{form of the solutions} (see also Example~\ref{a b non-integers} which shows that our conclusion is sharp).

\begin{theorem}
\label{main result}
Let $K$ be an algebraically closed field of characteristic $p$, let $N\in\N$,  let $V\subset \bG_m^N$ be an irreducible curve and $\Phi:X\lra X$ be a self-map both defined over $K$, and let $\alpha\in \bG_m^N(K)$. Then the set $S$ of all $n\in\N_0$ such that $\Phi^n(\alpha)\in V(K)$ is either a finite union of arithmetic progressions, or a finite union of sets of the form 
\begin{equation}
\label{form of solutions main result}
\left\{ap^{kn}+b\colon n\in\N_0\right\},
\end{equation}
for some $a,b\in\Q$ and $k\in\N_0$.
\end{theorem}

We note that if $k=0$ in \eqref{form of solutions main result}, then the corresponding set is a singleton; hence, both options (arithmetic progressions and sets as in \eqref{form of solutions main result}) allow for the possibility that $S$ may consist of finitely many elements (perhaps, along with finitely many infinite arithmetic progressions or finitely many sets as in \eqref{form of solutions main result}). Proposition~\ref{precise intersection curves} allows us to prove the more precise conclusion in Theorem~\ref{main result} compared to the one formulated in Conjecture~\ref{char p DML}. We show in the example below that indeed, it is possible in equation \eqref{form of solutions main result} from the conclusion of Theorem~\ref{main result} that $a$ and $b$ are \emph{not} contained in $\Z$.

\begin{example}
\label{a b non-integers}
Let $p$ be a prime number, let $V\subset \bG_m^2$ be the curve defined over $\F_p(t)$ given by the equation $tx+(1-t)y=1$, let $\Phi:\bG_m^2\lra \bG_m^2$ be the endomorphism given by $\Phi(x,y)=\left(t^{p^2-1}\cdot x, (1-t)^{p^2-1}\cdot y\right)$, and let $\alpha=(1,1)$. Then the set $S$ of all $n\in\N_0$ such that $\Phi^n(\alpha)\in V$ is 
$$\left\{\frac{1}{p^2-1}\cdot p^{2n} - \frac{1}{p^2-1}\colon n\in\N_0\right\}.$$
\end{example}

Trying to extend Theorem~\ref{main result} seems to be impossible at this moment. On one hand, if we work with an arbitrary quasiprojective variety $X$, then Conjecture~\ref{char p DML} is expected to be very difficult since even its counterpart in characteristic $0$ is only known to hold in special cases, and generally, when it holds, its proof involves methods which fail in positive characteristic, such as the so-called \emph{$p$-adic arc lemma} (see \cite[Chapter~4]{book} and its main application from \cite{Jason} to the dynamical Mordell-Lang conjecture for \'etale endomorphisms in characteristic $0$). On the other hand, even if $X$ is assumed to be $\bG_m^N$ and $\Phi$ is a group endomorphism, unless one assumes another hypothesis (either on $\Phi$ as in \cite[Chapter~13]{book}, or on $V$ as in Theorem~\ref{main result}), Conjecture~\ref{char p DML} is expected to be very difficult. Trying to solve it, even in the case of arbitrary subvarieties of $X=\bG_m^N$, it leads to difficult questions involving polynomial-exponential equations (see Remark~\ref{remark next}). Also, in Remark~\ref{remark next} we explain the difficulties one would have to overcome if one would try to extend Theorem~\ref{main result} even to the case of curves contained in an arbitrary semiabelian variety defined over a finite field.

We sketch below the plan of our paper. In Section~\ref{subsection F-sets} we discuss the classical Mordell-Lang problem for algebraic tori in characteristic $p$, by stating the results of Moosa-Scanlon \cite{Moosa-Scanlon} and of Derksen-Masser \cite{Derksen-Masser} and then  explaining its connections to Conjecture~\ref{char p DML}. In Section~\ref{section arithmetic sequences} we introduce the so-called $p$-arithmetic sequences, i.e., sets of the form \eqref{form of solutions main result} and discuss properties of these sequences including in the larger context of linear recurrence sequences. Then we  proceed to prove our main result in Section~\ref{section proof}.

\section{The Mordell-Lang problem in positive characteristic}
\label{subsection F-sets}


In the classical Mordell-Lang problem in characteristic $0$, Faltings \cite{Fal94} and Vojta \cite{Voj96} proved that if an irreducible variety $V$ of a semiabelian variety $X$ intersects a finitely generated subgroup  of $X$ in a Zariski dense set, then $V$ must be a translate of a connected algebraic subgroup of $X$. In particular, their result immediately yields a structure theorem for the intersection of any subvariety of $X$ with any finitely generated subgroup.  

Hrushovski \cite{Hrushovski} gave a complete description of all subvarieties $V$ of semiabelian varieties $X$ defined over fields $K$ of characteristic $p>0$ with the property that for some finitely generated $\Gamma\subset X(K)$, the intersection $V(K)\cap\Gamma$ is Zariski dense in $V$. However, since Hrushovski's result \cite{Hrushovski} is quite intricate, it does not give a structure theorem for the intersection of an arbitrary subvariety with an arbitrary finitely generated subgroup of $X$. In the case when $X$ is a semiabelian variety defined over a finite field, Moosa and Scanlon \cite{Moosa-Scanlon} (see also \cite{groups}) gave a complete description of the intersection $V(K)\cap\Gamma$ in terms of the so-called $F$-sets, where $F$ is the Frobenius endomorphism of $X$. In the case $X=\bG_m^N$, using a different method, Derksen and Masser \cite{Derksen-Masser} improved the result from \cite{Moosa-Scanlon} by showing that the intersection $V(K)\cap\Gamma$ can be computed effectively using heights. In order to state their structure theorem for the intersection $V(K)\cap\Gamma$ (see Theorem~\ref{Moosa-Scanlon theorem}), we introduce first the notion of $F$-sets (which also motivated the statement of Conjecture~\ref{char p DML}). 

\begin{definition} 
\label{definition F-sets}
Let $N\in\N$, let $K$ be an algebraically closed field of characteristic $p$, and let $\Gamma\subseteq \bG_m^N(K)$ be a finitely generated subgroup.
\begin{itemize}
\item[(a)]
By a \emph{product of $F$-orbits} in $\Gamma$ we mean a set of the form
$$C(\alpha_1,\dots,\alpha_m;k_1,\dots,k_m):=\left\{\prod_{j=1}^m \alpha_j^{p^{k_jn_j}} \colon n_j\in\mathbb{N}_0\right\}\subseteq\Gamma$$
where $\alpha_1,\dots,\alpha_m\in \bG_m^N(K)$ and $k_1,\dots,k_m\in\N_0$.
\item[(b)]
An \emph{$F$-set} in $\Gamma$ is a set of the form 
$C\cdot \Gamma '$ where $C$ is a product of $F$-orbits in $\Gamma$, and $\Gamma '\subseteq \Gamma$ is a subgroup, while in general, for two sets $A,B\subset \bG_m^N(K)$, the set $A\cdot B$ is simply the set $\{a\cdot b\colon a\in A\text{,  }b\in B\}$. 
\end{itemize}
\end{definition}

The motivation for using the name \emph{$F$-set} comes from the reference to the Frobenius corresponding to $\F_p$ (lifted to an endomorphism $F$ of $\bG_m^N$) since an $F$-orbit $C(\alpha;k)$ is simply the orbit of a point in $\bG_m^N(K)$ under $F^k$. 

We note that allowing for the possibility that some $k_i=0$ in the definition of a product of $F$-orbits, implicitly we allow a singleton be a product of $F$-orbits; this also explains why we do not need to consider cosets of subgroups $\Gamma '$ in the definition of an $F$-set. Furthermore, if $k_2=\cdots = k_m=0$ then the corresponding product of $F$-orbits is simply a translate of the single $F$-orbit $C(\alpha_1;k_1)$. Finally, note that the $F$-orbit $C(\alpha;k)$ for $k>0$ is finite if and only if $\alpha\in\bG_m^N\left(\overline{\mathbb{F}_p}\right)$.


\begin{theorem}[Moosa-Scanlon \cite{Moosa-Scanlon}, Derksen-Masser \cite{Derksen-Masser}]
\label{Moosa-Scanlon theorem}
Let $N\in\N$, let $K$ be an algebraically closed field of characteristic $p$, let $V\subset \bG_m^N$ be a subvariety defined over $K$ and let $\Gamma\subset \bG_m^N(K)$ be a finitely generated subgroup. Then $V(K)\cap \Gamma$ is an $F$-set contained in $\Gamma$.
\end{theorem}


If $V$ is an irreducible curve, one can deduce easily the following slightly more precise statement which will be used in the proof of Theorem~\ref{main result}.

\begin{corollary}
\label{precise intersection curves}
With the notation from Theorem~\ref{Moosa-Scanlon theorem}, if $V$ is an irreducible curve, then the intersection $V(K)\cap\Gamma$ is either finite, or a coset of an infinite subgroup of $\Gamma$, or a finite union of translates of $F$-orbits contained in $\Gamma$.
\end{corollary}

\begin{proof}
Assume $V(K)\cap\Gamma$ is infinite. Then, according to Theorem~\ref{Moosa-Scanlon theorem}, there exists an infinite $F$-set $C\cdot \Gamma '$ contained in the intersection $V(K)\cap\Gamma$. So, either the product $C$ of $F$-orbits is infinite, or the subgroup $\Gamma '$ of $\Gamma$ is infinite. 

If $\Gamma '$ is infinite, then $V$ must contain a coset $c \cdot\Gamma '$ (for some $c\in C$) and since $V$ is an irreducible curve, then $V=c \cdot H$, where $H$ is the Zariski closure of $\Gamma '$. Because $\Gamma '$ is an infinite group, we conclude that $H$ is a $1$-dimensional algebraic subtori; so, $V$ is a coset of this algebraic subgroup of $\bG_m^N$. In conclusion, the intersection $V\cap \Gamma$ must also be a coset of a subgroup of $\Gamma$; more precisely, it equals $c \cdot (H\cap\Gamma)$.

Assume now that $V$ is not a translate of an algebraic group, since in that case we get the desired conclusion; in particular, with the above notation, we also know that $\Gamma '$ is a finite subgroup of $\Gamma$. Then the product of $F$-orbits $C$ must be infinite. Assume $C$ contains the product of at least two infinite $F$-orbits; note that if there is only one infinite $F$-orbit in $C$, then we may re-write $C\cdot \Gamma '$ so that it is the union of finitely many translates of single $F$-orbits, as desired in our conclusion. Therefore, a translate $V'$ of $V$ must contain the product of two infinite $F$-orbits: 
$$C':=\left\{\alpha_1^{p^{k_1n_1}}\cdot \alpha_2^{p^{k_2n_2}}\colon n_1,n_2\in\N_0\right\},$$
for some $k_1,k_2\in\N$ and some $\alpha_1,\alpha_2\in \bG_m^N(K)\setminus \bG_m^N(\overline{\mathbb{F}_p})$. 
But then the stabilizer $W$ of $V'$ in $\bG_m^N$, i.e. the algebraic subgroup of $\bG_m^N$ consisting of  all points $x\in \bG_m^N$ such that $x\cdot V'\subset V'$ is positive-dimensional. Because $V'$ is a curve, then $V'$ must be a translate of $W$, and therefore $V$ is a translate of an algebraic group, contradicting our assumption. So, in this case it must be that $V(K)\cap\Gamma$ is a finite union of cosets of single $F$-orbits.


This concludes the proof of Proposition~\ref{precise intersection curves}.
\end{proof}


\subsection{Strategy for proving Theorem~\ref{main result}}

Here we discuss briefly the strategy for proving our main result. We have two main ingredients for the proof of Theorem~\ref{main result}: on one hand, we have Theorem~\ref{Moosa-Scanlon theorem} along with its slight strengthening from Corollary~\ref{precise intersection curves} in the case of irreducible curves, and on the other hand, we have a result combining theory of arithmetic sequences and linear algebra (Theorem~\ref{combinatorial theorem}). In order to prove Theorem~\ref{main result} we proceed as follows. With the notation as in Theorem~\ref{main result}, there exists a finitely generated group $\Gamma$ such that the orbit of $\alpha$ under $\Phi$ is contained inside $\Gamma$, and so, we obtain $V(K)\cap\cO_\Phi(\alpha)$ by intersecting first $V$ with $\Gamma$ and then with $\cO_{\Phi}(\alpha)$. We use Corollary~\ref{precise intersection curves} and then intersect each $F$-set appearing in $V(K)\cap\Gamma$ with $\cO_\Phi(\alpha)$ and apply  Theorem~\ref{combinatorial theorem}. So, essentially our argument splits into two cases:
\begin{itemize}
\item[(i)] given a coset $c\cdot \Gamma '$ of a subgroup of $\Gamma$, we show that the set of $n\in\N_0$ such that $\Phi^n(\alpha)\in c\cdot \Gamma '$ is a finite union of arithmetic progressions; and
\item[(ii)] given a translate $C$ of single $F$-orbit contained in $\Gamma$, we show that the set of $n\in\N_0$ such that $\Phi^n(\alpha)\in C$ is a  union of finitely many arithmetic progressions along with finitely many sets of the form \eqref{form of solutions main result}, i.e.,
$$
\left\{ap^{kn}+b\colon n\in\N_0\right\},
$$
for some given $a,b\in\Q$ and $k\in\N_0$. Such sets as the one above will be called \emph{$p$-arithmetic sequences} (see Definition~\ref{p-arithmetic sequences}). Then Proposition~\ref{prop:the more precise main result} delivers the more precise conclusion in Theorem~\ref{main result}.
\end{itemize}

Finally, we conclude this section by showing that for an irreducible curve $V$, if there exists an infinite arithmetic progression $\mathcal{P}$ such that $\Phi^n(\alpha)\in V$ for each $n\in\mathcal{P}$, then the set $S$ from the conclusion of Theorem~\ref{main result} is a finite union of infinite arithmetic progressions. Actually, as we will see in Proposition~\ref{prop:the more precise main result}, this conclusion holds in much higher generality than the hypotheses of Theorem~\ref{main result}.

\begin{proposition}
\label{prop:the more precise main result}
Let $X$ be a quasiprojective variety defined over an algebraically closed field $K$, let $\Phi:X\lra X$ be an endomorphism, let $V\subseteq X$ be an irreducible curve, let $\alpha\in X(K)$, and let
$$S=\{n\in\N_0\colon \Phi^n(\alpha)\in V(K)\}.$$
If $S$ contains an infinite arithmetic progression,  then $S$ is a finite union of arithmetic progression.
\end{proposition}

\begin{proof}
First we note that if $\alpha$ is preperiodic under the action of $\Phi$, i.e., its orbit $\cO_\Phi(\alpha)$ is finite, then the conclusion holds easily (see also \cite[Proposition~3.1.2.9]{book}). So, from now on, we assume $\alpha$ is not preperiodic under the action of $\Phi$.

By our assumption, we know that there exist $a,b\in\N$ such that $\{an+b\colon n\in\N_0\}\subseteq S$. Since $V$ is irreducible and $\alpha$ is not preperiodic, we conclude that $V$ is the Zariski closure of the set $\left\{\Phi^{an+b}(\alpha)\colon n\in\N_0\right\}$, and in particular, we get that $\Phi^a(V)\subseteq V$. Therefore, for each $i\in\{0,\dots, a-1\}$, if there exists some $j\in\N_0$ such that $j\equiv i\pmod{a}$ and moreover, $\Phi^j(\alpha)\in V$, then $\{an+j\colon n\in\N_0\}\subset S$. In conclusion, $S$ itself is a finite union of arithmetic progressions. 
\end{proof}

\begin{remark}
\label{a single infinite arithmetic progression}
Actually, in Proposition~\ref{prop:the more precise main result}, under the assumption that $\alpha$ is not preperiodic under $\Phi$, if we let $a_0$ be the smallest positive integer such that there exist $j_1,j_2\in\N_0$ such that $j_2-j_1=a_0$ and moreover, $\Phi^{j_1}(\alpha),\Phi^{j_2}(\alpha)\in V$, we obtain that the set $S$ is itself a \emph{single} infinite arithmetic progression of common difference $a_0$. However, this more precise statement is not needed later in our proof.
\end{remark}


\section{Arithmetic sequences}
\label{section arithmetic sequences}

In this section we prove various useful results regarding linear recurrence sequences (see Definition~\ref{definition linear recurrence sequence}). Two special types of linear recurrence sequences which appear often in our paper are arithmetic progressions and $p$-arithmetic sequences (see Definition~\ref{p-arithmetic sequences}).

\subsection{Linear recurrence sequences}
\label{section:linear recurrence sequence}

We define next linear recurrence sequences which are essential in our proof.
\begin{definition}
\label{definition linear recurrence sequence}
A linear recurrence sequence is a sequence $\{u_n\}_{n\in\N_0}\subset \C$ with the property that there exists $m\in\N$ and there exist $c_0,\dots, c_{m-1}\in\C$ such that for each $n\in\N_0$ we have
\begin{equation}
\label{equation definition linear recurrence sequence}
u_{n+m}+c_{m-1}u_{n+m-1}+\cdots + c_1u_{n+1}+c_0u_n=0.
\end{equation}
\end{definition}

If $c_0=0$ in \eqref{equation definition linear recurrence sequence}, then we may replace $m$  by $m-k$ where $k$ is the smallest positive integer such that $c_k\ne 0$; then the sequence $\{u_n\}_{n\in\N_0}$ satisfies the linear recurrence relation
\begin{equation}
\label{equation definition linear recurrence sequence 2}
u_{n+m-k}+c_{m-1}u_{n+m-1-k}+\cdots + c_ku_{n}=0,
\end{equation}
for all $n\ge k$. In particular, if $k=m$, then the sequence $\{u_n\}_{n\in\N_0}$ is eventually constant, i.e., $u_n=0$ for all $n\ge m$. 

Assume now that $c_0\ne 0$ (which may be achieved at the expense of re-writing the recurrence relation as in \eqref{equation definition linear recurrence sequence 2}); then there exists a closed form for expressing $u_n$ for all $n$ (or at least for all $n$ sufficiently large if one needs to re-write the recurrence relation as in \eqref{equation definition linear recurrence sequence 2}). The characteristic roots of a linear recurrence sequence as in \eqref{equation definition linear recurrence sequence} are the roots of the equation
\begin{equation}
\label{characteristic equation}
x^m+c_{m-1}x^{m-1}+\cdots +c_1x+c_0=0.
\end{equation}
We let $r_i$ (for $1\le i\le s$) be the (nonzero) roots of the equation \eqref{characteristic equation}; then there exist polynomials $P_i(x)\in\C[x]$ such that for all  $n\in\N_0$, we have
\begin{equation}
\label{general formula linear recurrence sequence}
u_n=\sum_{i=1}^s P_i(n)r_i^n.
\end{equation}
In general, as explained above, for an arbitrary linear recurrence sequence, the formula \eqref{general formula linear recurrence sequence} holds for all $n$ sufficiently large (more precisely, for all $n\ge m$ with the notation from \eqref{equation definition linear recurrence sequence}); for more details on linear recurrence sequences, we refer the reader to the chapter on linear recurrence sequences written by Schmidt in the book \cite{Sch03}.

It will be convenient for us later on in our proof to consider linear recurrence sequences which are given by a formula such as the one in \eqref{general formula linear recurrence sequence} (for $n$ sufficiently large) for which the following two properties hold:
\begin{enumerate}
\item[(i)] if some $r_i$ is a root of unity, then $r_i=1$; and  
\item[(ii)] if $i\ne j$, then $r_i/r_j$ is not a root of unity.
\end{enumerate}   
Such linear recurrence sequences given by formula \eqref{general formula linear recurrence sequence} and satisfying properties~(i)-(ii) above are called \emph{non-degenerate}. Given an arbitrary linear recurrence sequence, we can always split it into finitely many linear recurrence sequences which are all non-degenerate; moreover, we can achieve this by considering instead of one sequence $\{u_n\}$, finitely many sequences which are all of the form $\{u_{nM+\ell}\}$ for a given $M\in\N$ and for $\ell=0,\dots, M-1$. Indeed, assume some $r_i$ or some $r_i/r_j$ is a root of unity, say of order $M$; then for each $\ell=0,\dots, M-1$ we have that 
\begin{equation}
\label{re-writing u n}
u_{nM+\ell}=\sum_{i=1}^s P_i(nM+\ell)r_i^{\ell} (r_i^M)^n
\end{equation}
and moreover, we can re-write the formula \eqref{re-writing u n} for $u_{nM+\ell}$ by collecting the powers $r_i^M$ which are equal and thus achieve a non-degenerate linear recurrence sequence $\{u_{nM+\ell}\}$.

The following famous theorem of Skolem \cite{Skolem} (later generalized by Mahler \cite{Mahler} and Lech \cite{Lech})  will be used throughout our proof.
\begin{proposition}
\label{Skolem result}
Let $\{u_n\}_{n\in\N_0}\subset \C$ be a linear recurrence sequence, and let $c\in\C$. Then the set $S$ of all $n\in\N_0$ such that $u_n=c$ is a finite union of arithmetic progressions; moreover, if $\{u_n\}$ is a non-degenerate linear recurrence sequence, then the set $S$ is finite. 
\end{proposition}

\subsection{A special type of linear recurrence sequences}

We start by introducing the following notation.
\begin{definition}
\label{p-arithmetic sequences}
Let $p$ be a prime number. We call a $p$-arithmetic sequence a set of the form
$$\{ap^{kn}+b\colon n\in\N_0\},$$
for some $a,b\in\Q$ and $k\in\N_0$.  
\end{definition}

Note that $p$-arithmetic sequences, similar to arithmetic progressions are allowed to be sets consisting of a single element (which would be the case when $a=0$ or $k=0$ in Definition~\ref{p-arithmetic sequences}). Next we will prove several results regarding $p$-arithmetic sequences  will be used later in the proof of Theorem~\ref{main result}.

\begin{proposition}
\label{prop:p-arithmetic and arithmetic}
Let $p$ be a prime number. 
\begin{enumerate}  
\item[(a)] The intersection of an arithmetic progression with a $p$-arithmetic sequence is a finite union of $p$-arithmetic sequences. 
\item[(b)] Let $a,b\in\Q$ and let $k\in\N_0$. Then the set $\{ap^{kn}+b\colon n\in\N_0\}\cap \N_0$
is itself a finite union of $p$-arithmetic sequences.
\end{enumerate}
\end{proposition}

\begin{proof}
The two parts are actually equivalent; the fact that part~(b) is a special case of part~(a) is obvious, but also the converse holds. Indeed, letting $a,b,c,d\in\Q$ and $k\in\N_0$, the set of all $cn+d$ which can be written as $ap^{km}+b$ for some $m\in\N_0$ is a $p$-arithmetic sequence once we know that the set
$$\left\{\frac{a}{c}\cdot p^{mk}+\frac{b-d}{c}\colon m\in\N_0\right\}\cap\N_0$$
is a $p$-arithmetic sequence. On the other hand, part~(b) is a simple exercise using basic properties of congruences.  
\end{proof}

The following proposition is a standard consequence of Siegel's celebrated theorem \cite{Siegel} regarding $S$-integral points on curves of positive genus; we include its proof for the sake of completeness.
\begin{proposition}
\label{prop:Siegel}
Let $s$ be a nonzero algebraic number which is not a root of unity, and let $P(x)\in\Qbar[x]$ be a polynomial of degree $d$. If there exist infinitely many pairs $(m,n)\in \N_0\times\N_0$ such that $P(n)=s^m$, then there exist $a,b\in\Qbar$ such that $P(x)=a(x-b)^d$.
\end{proposition}

\begin{proof}
We argue by contradiction, and thus assume $P(x)$ has at least two distinct roots. We let $r_1,\dots, r_\ell$ be the distinct roots of $P(x)$ and thus write
$$P(x)=a\cdot \prod_{i=1}^\ell (x-r_i)^{e_i}\text{ for some }e_i\in\N\text{ and }a\in\Qbar.$$
We let $q$ be an odd prime number larger than each $e_i$, and we let $a_1,s_1\in \Qbar$ such that $a_1^q=a$ and $s_1^q=s$. 
We let $K$ be a number field containing $a_1$, $s_1$ and also each $r_i$ for $i=1,\dots, \ell$. We let $R$ be the ring of algebraic integers contained in $K$. We let $S$ be a finite set of (nonzero) prime ideals of $R$ such that all of the following properties hold:
\begin{itemize}
\item[(i)] $a_1$, $s_1$ and each $r_i$ are contained in the subring $R_S$ of $S$-integers contained in $K$ (i.e., the ring of all elements of $K$ which are integral at all places away from the set $S$).
\item[(ii)] $r_i-r_j\in R_S^*$ for each $1\le i<j\le \ell$. 
\item[(iii)] $R_S$ is a PID.
\end{itemize}
For each $(m,n)\in\N_0\times\N_0$ such that $P(n)=s^m$, conditions~(i)-(ii) yield that each ideal $(n-r_i)$ is the $q$-th power of an ideal in $R_S$ and moreover, using condition~(iii) above, we conclude that there exist units $u_i\in R_S^*$ and also elements $z_i\in R_S$ such that $n-r_i=u_iz_i^q$ for each $i=1,\dots, \ell$. Furthermore, because $R_S^*$ is a finitely generated abelian group and therefore $R_S^*/(R_S^*)^q$ is a finite group, we may and do assume that the units $u_i$ belong to a finite set $U$ of units in the ring $R_S$.  So, by the pigeonhole principle (because there exist infinitely many solutions $(m,n)\in\N_0\times\N_0$ for the equation $P(n)=s^m$), there exist $u,v\in U$ for which there exist infinitely many $(z,w)\in R_S\times R_S$ such that 
\begin{equation}
\label{equation curve Siegel}
uz^q-vw^q=r_1-r_2.
\end{equation}
Because $q\ge 3$ and $r_1\ne r_2$, the equation \eqref{equation curve Siegel} in the variables $(z,w)$ yields a curve of positive genus which cannot contain infinitely many $S$-integral points, by Siegel's theorem \cite{Siegel}. This concludes our proof that $P(x)$ has only one distinct root.
\end{proof}

The following fact is an easy consequence of Proposition~\ref{prop:Siegel}. 

\begin{corollary}
\label{cor:Siegel}
Let $p$ be a prime number, let $k\in\N_0$, let $a,b\in\Q$, and let $P\in\Q[x]$ of degree $d\ge 1$. The set $T$ consisting of all $n\in\N_0$ for which there exists $m\in\N_0$ such that $P(n)=b+ap^{mk}$ is a finite union of $p$-arithmetic sequences.
\end{corollary}

\begin{proof}
First we note that if $T$ is finite, then the conclusion is obvious; so, from now on, we assume $T$ is infinite. In particular, this yields that both $m$ and $a$ are nonzero. 


By Proposition~\ref{prop:Siegel} we obtain that $P_1(x):=(P(x)-b)/a$ has only one root, which must be thus a rational number (since $P\in\Q[x]$); so, we let $P_1(x)=A(x-c)^d$ for some $c,A\in \Q$. Since there exist infinitely many $n\in\N_0$ such that $P_1(n)=p^{km}$, we conclude that there exists some $\ell\in\Z$ such that $Ap^{-\ell}$ is the $d$-th power of a rational number; hence $P_1(x)=p^{\ell}(ex-f)^d$ for some $e,f\in\Q$ and then the conclusion follows easily (see also part~(b) of Proposition~\ref{prop:p-arithmetic and arithmetic}).     
\end{proof}

\begin{proposition}
\label{prop:solutions of p-arithmetic sequences}
Let $\{u_n\}_{n\in\N_0}\subset \Z$ be a linear recurrence sequence, let $p$ be a prime number, let $k\in\N_0$, and let $a,b\in\Q$. Then the set 
$$T=\{n\in\N_0\colon u_n=b+ap^{km}\text{ for some }m\in\N_0\}$$
is a finite union of arithmetic progressions along with a finite union of $p$-arithmetic sequences.
\end{proposition}

\begin{proof}
If $k=0$ or $a=0$, then the conclusion follows immediately from Proposition~\ref{Skolem result}. So, from now on, we assume that both $a$ and $k$ are nonzero. Also, we assume the set $T$ is infinite (otherwise, the conclusion is obvious).

Because $\{u_n\}_{n\in\N_0}$ is a linear reccurence sequence, then at the expense of replacing $\{u_n\}$ by finitely many linear recurrence sequences (obtained by replacing $n$ by a suitable arithmetic progression, as in Section~\ref{section:linear recurrence sequence}), we may assume the sequence $\{u_n\}$ is non-degenerate and for (at least sufficiently large) $n\in\N_0$, we have that $u_n=\sum_{i=1}^s P_i(n)r_i^n$ for some distinct $r_i\in\Qbar^*$ and some nonzero polynomials $P_i(x)\in\Qbar[x]$. Furthermore, if some $r_i$ is a root of unity, then $r_i=1$; also, if $i\ne j$ then $r_i/r_j$ is not a root of unity.


With the above notation and convention for our linear recurrence sequence being non-degenerate, if $s=1$ and $r_1=1$, i.e., $u_n=P_1(n)$ for some polynomial $P_1\in\Q[x]$, then the conclusion follows from Corollary~\ref{cor:Siegel}. So, from now on, assume either $s>1$, or $u_n=P_1(n)r_1^n$ for some non-root of unity $r_1$. Since $v_m:=b+ap^{km}$ is itself a linear recurrence sequence, and both $\{u_n\}$ and $\{v_m\}$ are non-degenerate linear recurrence sequences which are not of the form $\alpha_0^nQ(n)$ for some root of unity $\alpha_0$, and since there exist infinitely many solutions $(m,n)$ for the equation $u_n=v_m$, then \cite[Theorem~11.2,~p.~209]{Sch03} yields that $u_n=b+Q(n)(\pm p^\gamma)^{n}$ for some $\gamma\in\N$ and some $Q(x)\in\Q[x]$. We get this conclusion because \cite[Theorem~11.2,~p.~209]{Sch03} yields that $u_n$ and $v_m$ must be related (see also \cite[(11.3)]{Sch03}) and therefore, $u_n$ has exactly one characteristic root $r$ which is not equal to $1$ and furthermore $|r|=p^\gamma$ for some nonzero $\gamma\in\Q$. Since $u_n\in\Z$, then $r=\pm p^\gamma$ and $\gamma\in\N$; so, $u_n=Q_0(n)+Q(n)r^n$. Finally, \cite[Theorem~11.2,~p.~209]{Sch03} yields that $Q_0$ must be the constant polynomial equal to $b$.  

Then solving $u_n=v_m$ yields two possibilities: either $n$ is even, or $n$ is odd. Since the latter case follows almost identically as the former case, we focus only on the case $n$ is even. So, we are solving the equation  
$$Q(2n) p^{2\gamma n}=ap^{km}$$ 
and therefore $Q(2n)=ap^{km-2\gamma n}$. If $Q(x)$ is a constant polynomial, then we get that $km-2\gamma n$ must also be constant and therefore, the set of such $n\in\N_0$ is an arithmetic progression. Now, if $\deg(Q)\ge 1$ then $|Q(2n)|>1$ for $n$ sufficiently large and therefore, except for finitely many solutions, we must have that $km \ge  2\gamma n$. Then the set of $n\in\N_0$ such that $Q(2n)$ belongs to the the $p$-arithmetic sequence $\{ap^j\}_{j\in\N_0}$ is itself a finite union of $p$-arithmetic sequences, as proven in Corollary~\ref{cor:Siegel}. 
This concludes the proof of Proposition~\ref{prop:solutions of p-arithmetic sequences}.
\end{proof}

\begin{proposition}
\label{prop:intersections of p-arithmetic sequences}
The intersection of two $p$-arithmetic sequences is a finite union of $p$-arithmetic sequences. 
\end{proposition}

\begin{proof}
The result is a very simple consequence of Vojta's powerful theorem \cite{Voj96}; however, we can easily prove it directly, as shown in the next few lines. 

So, given $a_i,b_i\in\Q$ and $k_i\in\N_0$ for $i=1,2$, we look for solutions in $(n_1,n_2)\in\N_0\times\N_0$ of the equation 
\begin{equation}
\label{a b k i}
b_1+a_1p^{k_1n_1}=b_2+a_2p^{k_2n_2}. 
\end{equation}
Now, if some $a_i$ or some $k_i$ equals $0$, then the conclusion is immediate. So, from now on, we assume each $a_i$ and each $k_i$ is nonzero. Also, we assume there exist infinitely many solutions to equation \eqref{a b k i} since otherwise the conclusion is obvious.

The existence of infinitely many  solutions to \eqref{a b k i} coupled with \cite[Proposition~11.2,~p.~209]{Sch03} yields that $b_1=b_2$. Then there must be some $\ell\in\Z$ such that $\frac{a_2}{a_1}=p^\ell$. So, equation \eqref{a b k i} reads now: $k_1n_1=k_2n_2+\ell$. The conclusion of Proposition~\ref{prop:solutions of p-arithmetic sequences} follows readily. 
\end{proof}


\section{Proof of our main result}
\label{section proof}

The main ingredient of our proof is the following result.
\begin{theorem}
\label{combinatorial theorem}
Let $(G,+)$ be a finitely generated abelian group, let $y\in G$, and let $\Phi_0:G\lra G$ be a group homomorphism with the property that there exist $\ell\in\N$ and also there exist $c_0,\dots, c_{\ell-1}\in\Z$ such that
\begin{equation}
\label{equation Phi 0}
\Phi_0^\ell(x)+c_{\ell-1}\Phi_0^{\ell-1}(x)+\cdots + c_1\Phi_0(x)+c_0x=0\text{ for all }x\in G.
\end{equation}
Let $\Phi:G\lra G$ be given by $\Phi(x)=y+\Phi_0(x)$ for each $x\in G$, and also let $\alpha\in G$.
\begin{enumerate}
\item[(A)] Let $R\in G$ and let $H$ be a subgroup of $G$. Then the set $S$ of all $n\in\N_0$ such that $\Phi^n(\alpha)\in (R+H)$ is a finite union of arithmetic progressions. 
\item[(B)] Let $p$ be a prime number, let $k\in\N_0$, let $R_1,R_2\in G$, and let
$$C=\{R_1+p^{nk}R_2\colon n\in\N_0\}.$$
Then the set $S$ of all $n\in\N_0$ such that $\Phi^n(\alpha)\in C$ is a finite union of $p$-arithmetic sequences along with a finite union of arithmetic progressions. 
\end{enumerate}
\end{theorem}

\begin{proof}
We first prove the following important claim which will be key for both parts of our proposition.
\begin{claim}
\label{claim:linear recurrence sequence}
There exist linear recurrence sequences $\{u_{i,n}\}_{n\in\N_0}\subset \Z$ for $1\le i\le \ell$ and $\{v_{i,n}\}_{n\in\N_0}\subset \Z$ for $0\le i\le \ell-1$ such that for each $n\in\N_0$, we have
$$
\Phi^n(\alpha)=\sum_{i=1}^\ell u_{i,n}\left(\sum_{j=0}^{i-1} \Phi_0^j(y)\right) + \sum_{i=0}^{\ell-1} v_{i,n}\Phi_0^i(\alpha).$$
\end{claim}

\begin{proof}[Proof of Claim~\ref{claim:linear recurrence sequence}.]
For each $n\in\N_0$ we have that
$$\Phi^n(\alpha)=\left(\sum_{i=0}^{n-1} \Phi_0^i(y)\right) + \Phi_0^n(\alpha).$$
Due to \eqref{equation Phi 0}, we get that the sequence of points $\{\Phi_0^n(\alpha)\}_{n\in\N_0}$ satisfies the linear recurrence relation:
\begin{equation}
\label{equation Phi 0 1}
\Phi_0^{n+\ell}(\alpha) + \sum_{i=0}^{\ell-1} c_i \Phi_0^{n+i}(\alpha)=0.
\end{equation}
So, there exist $\ell$ linear recurrence sequences $\{v_{i,n}\}_{n\in\N_0}\subset \Z$, each one satisfying the linear reccurence formula:
\begin{equation}
\label{equation v}
v_{i,n+\ell} + \sum_{j=0}^{\ell-1}c_jv_{i,n+j}=0\text{ for each }n\in\N_0\text{ and for each }0\le i\le \ell-1,
\end{equation}
such that 
\begin{equation}
\label{equation Phi 0 2}
\Phi_0^n(\alpha) = \sum_{i=0}^{\ell-1}v_{i,n}\Phi_0^i(\alpha).
\end{equation}
We let $Q_n:=\sum_{i=0}^{n-1} \Phi_0^i(y)\in G$ for $n\ge 1$, and also let $Q_0:=0$. Because of the linear recurrence formula \eqref{equation Phi 0 1} satisfied by the sequence $\{\Phi_0^n(\alpha)\}_{n\in\N_0}$, we conclude that the sequence $\{Q_n\}_{n\in\N_0}$ satisfies the following linear recurrence sequence:
\begin{equation}
\label{equation Q n}
Q_{n+\ell+1} + (c_{\ell-1}-1)Q_{n+\ell}  + \cdots  + (c_0-c_1)Q_{n+1}-c_0Q_n=0.
\end{equation}
The linear recurrence relation \eqref{equation Q n} yields the existence of the linear recurrence sequences $\{u_{i,n}\}_{n\in\N_0}\subset \Z$ for $1\le i\le \ell$, each of them satisfying the recurrence relation:
\begin{equation}
\label{equation u}
u_{i,n+\ell+1} + (c_{\ell-1}-1)u_{i,n+\ell}  + \cdots + (c_0-c_1)u_{i,n+1} - c_0u_{i,n}=0 
\end{equation}
such that 
\begin{equation}
\label{equation Phi 0 3}
Q_n=\sum_{i=1}^\ell u_{i,n}Q_i.
\end{equation}
This concludes the proof of Claim~\ref{claim:linear recurrence sequence}.
\end{proof}
Now we proceed to proving the two parts of our proposition.

{\bf Part (A).} Since $G$ is a finitely generated abelian group, we know that $G$ is isomorphic to a direct sum of a finite subgroup $G_0$ with a subgroup $G_1$ which is isomorphic to $\Z^r$ for some $r\in\N_0$. We let $\{P_1,\dots, P_r\}$ be a $\Z$-basis for $G_1$. We proceed with the notation from Claim~\ref{claim:linear recurrence sequence} and write
$$\Phi^n(\alpha)=\sum_{i=1}^\ell u_{i,n}Q_i + \sum_{i=0}^{\ell-1}v_{i,n}\Phi_0^i(\alpha),$$ 
and then we write each $Q_i$ (for $i=1,\dots, \ell$) as $Q_{i,0}+\sum_{j=1}^r b_{i,j}P_j$ with $Q_{i,0}\in G_0$ and each $b_{i,j}\in\Z$, and also write each $\Phi_0^i(\alpha)$ (for $i=0,\dots, \ell-1$) as $T_{i,0}+\sum_{j=1}^r a_{i,j}P_j$ with $T_{i,0}\in G_0$ and each $a_{i,j}\in\Z$. We also write 
$$R=R_0+\sum_{j=1}^r d_jP_j\text{ with }R_0\in G_0\text{ and each }d_j\in\Z.$$
We will write the conditions satisfied by the $u_{i,n}$ and $v_{i,n}$ in order for $\Phi^n(\alpha)\in (R + H)$, or equivalently that $\Phi^n(\alpha) - R\in H$.  In order to do this, we let $H_1=H\cap G_1$, which is also a free abelian finitely generated group, of rank $s\le r$. Then $H$ is a union of finitely many cosets $U_i + H_1$ with $U_i\in G_0$. Now, for a given $U:=U_i$, the condition that $\Phi^n(\alpha)-R\in (U + H_1)$, i.e., that $\Phi^n(\alpha) - R - U \in H_1$ can be written as follows:
\begin{equation}
\label{the condition 1}
\sum_{i=1}^\ell u_{i,n}Q_{i,0} + \sum_{i=0}^{\ell-1} v_{i,n}T_{i,0} =R_0 + U
\end{equation}
and
\begin{equation}
\label{the condition 2}
\sum_{i=1}^\ell u_{i,n}\sum_{j=1}^r b_{i,j}P_j + \sum_{i=0}^{\ell-1} v_{i,n}\sum_{j=1}^r a_{i,j}P_j - \sum_{j=1}^r d_jP_j\in H_1.
\end{equation}
Now, the set of all $n\in\N_0$ which satisfy condition \eqref{the condition 1} is a finite union of arithmetic progressions (with the understanding, as always, that a singleton is an arithmetic progression of common difference equal to $0$). Indeed, this claim holds since each of the points $Q_{i,0}$, $T_{i,0}$, $R_0$ and $U$ belong to a finite group, and so they all have finite order bounded by $|G_0|$. Furthermore, it is a standard fact that any linear recurrence sequence is preperiodic modulo any given positive integer (see also the proof of \cite[Claim~3.3]{groups}). 

So, in order to finish the proof of part~(A) of Theorem~\ref{combinatorial theorem}, it suffices to prove that the set of all $n\in\N_0$ satisfying condition \eqref{the condition 2} is also a finite union of arithmetic progressions; clearly, an intersection of two arithmetic progressions is also an arithmetic progression and this would finish our argument for part~(A) of our theorem. Collecting coefficients of each $P_j$ for $j=1,\dots, r$, we get
$$w_{j,n}:=\sum_{i=1}^\ell b_{i,j}u_{i,n} + \sum_{i=0}^{\ell-1} a_{i,j} v_{i,n} - d_j,$$
and since both $u_{i,n}$ and $v_{i,n}$ are linear recurrence sequences, we conclude that also each $w_{j,n}$ is a linear recurrence sequence. Now, according to \cite[Claim~3.4,~Definition~3.5,~Subclaim~3.6]{groups}, the condition~\eqref{the condition 2} is equivalent with a system formed by finitely many equations of the form
\begin{equation}
\label{the condition 3}
\sum_{j=1}^r e_j w_{j,n} = 0
\end{equation} 
for some integers $e_j$, and of the form
\begin{equation}
\label{the condition 4}
\sum_{j=1}^r f_j w_{j,n}\equiv 0\pmod{M},
\end{equation}
for some other integers $f_j$ and also some $M\in\N$. Since itself, the sequence $\{w_n\}_{n\in\N_0}$ given by $w_n:=\sum_{j=1}^r e_jw_{j,n}$ (or respectively $\sum_{j=1}^r f_jw_{j,n}$) is a linear recurrence sequence, we immediately derive that the set of all $n\in\N_0$ satisfying either one of conditions~\ref{the condition 3} or \eqref{the condition 4} is a finite union of arithmetic progressions (see also Proposition~\ref{Skolem result}).

{\bf Part (B).} We proceed with the notation  from the proof of {\bf Part (A)}   for writing $\Phi^n(\alpha)=\sum_{i=1}^\ell u_{i,n}Q_i+\sum_{i=0}^{\ell-1}v_{i,n}\Phi_0^i(\alpha)$ in terms of the basis $\{P_1,\dots, P_r\}$ of $G_1$. We also write $R_1$ and $R_2$ in terms of elements of the finite group $G_0$ and of the basis $\{P_1,\dots, P_r\}$ of $G_1$, as follows:
$$R_i=R_{i,0}+\sum_{j=1}^r d_{i,j}P_j,$$
for some $R_{i,0}\in G_0$ and integers $d_{i,j}$, for $j=1,\dots, r$ and for $i=1,2$. Then $\Phi^n(\alpha)\in C$ if and only if both of the following two conditions are met:
\begin{equation}
\label{the condition 5}
\sum_{i=1}^\ell u_{i,n}Q_{i,0} + \sum_{i=0}^{\ell-1} v_{i,n} T_{i,0} = R_{1,0} + p^{mk} R_{2,0},
\end{equation}
and
\begin{equation}
\label{the condition 6}
\sum_{i=1}^\ell u_{i,n}\sum_{j=1}^r b_{i,j}P_j + \sum_{i=0}^{\ell-1}v_{i,n}\sum_{j=1}^r a_{i,j}P_j = \sum_{j=1}^r \left(d_{1,j}+p^{mk}d_{2,j}\right)P_j,
\end{equation}
for some $m\in\N_0$. Note that for any given $m\in\N_0$, the set of solutions $n\in\N_0$ for the system of equations \eqref{the condition 5} and \eqref{the condition 6} is a union of finitely many arithmetic progressions because fgor fixed $m$  we can easily argue as in part~(A). So, from now on, we may assume $m$ is sufficiently large. 

Now, since $R_{2,0}$ is a torsion point, then there exists $s\in\N$ such that for each $e=0,\dots, s-1$ and for each $m$ in the  arithmetic progression $\{e + sj\}$ (for $j$ sufficiently large)  the right-hand side of equation \eqref{the condition 5} is constant, say that it is equal to $R^{(s,e)}$. Replacing  $m$ by $e+sm$ in both equations \eqref{the condition 5} and \eqref{the condition 6} leads to the following equations:
\begin{equation}
\label{the condition 5'}
\sum_{i=1}^\ell u_{i,n}Q_{i,0} + \sum_{i=0}^{\ell-1} v_{i,n} T_{i,0} = R^{(s,e)}
\end{equation}
and
\begin{equation}
\label{the condition 6'}
\sum_{j=1}^r w_{j,n}P_j=\sum_{j=1}^r \left(d_{1,j} + p^{ke}d_{2,j}\cdot p^{msk}\right)P_j,
\end{equation}
where $w_{j,n}:=\sum_{i=1}^\ell u_{i,n}b_{i,j} + \sum_{i=0}^{\ell-1} v_{i,n}a_{i,j}$ is a linear recurrence sequence of integers for each $j=1,\dots, r$ because each of the $v_{i,n}$ and $u_{i,n}$ are linear recurrence sequences of integers (and also $b_{i,j}$ and $a_{i,j}$ are integers). We let $b_j:=d_{1,j}$ and $a_j:=p^{ke}d_{2,j}$ for each $j=1,\dots, r$. 

Since each integer is of the form $sk+e$ for some $e\in\{0,\dots, s-1\}$, it suffices to show that for given $e$ and $s$, the set of $n\in\N_0$ satisfying simultaneously equations \eqref{the condition 5'} and \eqref{the condition 6'} is a finite union of arithmetic progressions along with a finite union of $p$-arithmetic sequences.

Because each $Q_{i,0}$ and also each $T_{i,0}$ are torsion elements, while $\{v_{i,n}\}$ and $\{u_{i,n}\}$ are linear recurrence sequences, we conclude that equation \eqref{the condition 5'} is satisfied by integers $n$ which belong to finitely many arithmetic progressions (note that linear recurrence sequences of integers are preperiodic modulo any given integer). Using Proposition~\ref{prop:p-arithmetic and arithmetic}, it suffices to prove that the set of all $n\in\N_0$ satisfying \eqref{the condition 6'} is a finite union of $p$-arithmetic sequences along with a finite union of arithmetic progressions. 

Since equation \eqref{the condition 6'} holds inside a free abelian group whose $\Z$-basis is given by the $P_j$'s, then we obtain that \eqref{the condition 6} is equivalent with the simultaneous solution of the following equations:
\begin{equation}
\label{the condition 7}
w_{j,n}=b_j+p^{msk}a_j\text{ for each $j=1,\dots, r$.}
\end{equation}  

Now, if $k=0$, then each equation \eqref{the condition 7} is of the form $w_{j,n}=a_j+b_j$ and by Proposition~\ref{Skolem result}, the set of solutions $n$ is each time a finite union of arithmetic progressions; then clearly the set of $n\in\N_0$ satisfying simultanouesly all equations \eqref{the condition 7} is also a finite union of arithmetic progressions. So, from now on, assume $k>0$ (also note that $s>0$).
 
We have two types of equations of the form \eqref{the condition 7}: either  $a_j$ equals $0$, or not. If $a_j=0$, then equation \eqref{the condition 7} reduces to solving the equation $w_{j,n}=b_j$. If $a_j\ne 0$, then letting $w_{j,n}':=(w_{j,n}-b_j)/a_j$, we get another linear recurrence sequence; moreover, equation \eqref{the condition 7} yields $w'_{j,n}=p^{skm}$. So, the system of equations \eqref{the condition 7} splits into finitely many equations of the form
\begin{equation}
\label{the condition 8}
w_{j,n}=b_j 
\end{equation}
and also finitely many equations of the form 
\begin{equation}
\label{the condition 9}
w'_{j,n}=p^{skm} 
\end{equation}
for the same $m\in\N_0$. So, at the expense of re-labelling the linear recurrence sequences $\{w_{j,n}\}$ and $\{w'_{j,n}\}$, we may write the equations \eqref{the condition 8} and \eqref{the condition 9} combined as:
\begin{eqnarray}
\label{the condition 10}
w'_{1,n}=p^{skm}\text{ for some }m\in\N_0\\
\label{the condition 11}
w'_{1,n}=w'_{2,n}=\cdots = w'_{i_0,n}\text{ and}\\
\label{the condition 12}
w_{j,n}=b_j\text{ for }i_0<j\le r.
\end{eqnarray}
The solutions $n$ to equation \eqref{the condition 10} form finitely many $p$-arithmetic sequences along with finitely many arithmetic progressions (by Proposition~\ref{prop:solutions of p-arithmetic sequences}), while the solutions $n$ to each equation \eqref{the condition 11} and \eqref{the condition 12} form finitely many arithmetic progressions (by Proposition~\ref{Skolem result}). Another application of Proposition~\ref{prop:p-arithmetic and arithmetic} along with Proposition~\ref{prop:intersections of p-arithmetic sequences} finishes the proof of  Theorem~\ref{combinatorial theorem}.
\end{proof}

Our main result is now a simple consequence of Theorem~\ref{combinatorial theorem}.
\begin{proof}[Proof of Theorem~\ref{main result}.]
Since $\Phi$ is a self-map of $\bG_m^N$, then there exists $y\in \bG_m^N(K)$ and there exists a group endomorphism $\Phi_0$ of $\bG_m^N$ such that $\Phi(x)=y\cdot \Phi_0(x)$ for each $x\in\bG_m^N(K)$.  
We let $\Gamma_0$ be the subgroup of $\bG_m(K)$ spanned by all the coordinates of $\alpha$ and also all coordinates of $y$, and then we let $\Gamma=\Gamma_0^N\subset \bG_m^N(K)$ be its $N$-th cartesian power. Clearly, $\cO_\Phi(\alpha)\subset \Gamma$. Applying Corollary~\ref{precise intersection curves}, the intersection $V(K)\cap\Gamma$ is either a finite union of cosets of subgroups of $\Gamma$, or a finite union of sets of the form $\left\{b_i\cdot a_i^{p^{nk}}\right\}_{n\in\N_0}$ (for some $a,b\in\bG_m^N(K)$ and some $k\in\N_0$) contained in $\Gamma$. At the expense of replacing $\Gamma$ by a larger finitely generated group, we may even assume that the base points $a_i$ and $b_i$ above are also contained in $\Gamma$. Then the conclusion follows from Theorem~\ref{combinatorial theorem}; note also that Proposition~\ref{prop:the more precise main result} yields that if $S$ contains an infinite arithmetic progression, then $S$ itself is a finite union of arithmetic progressions. 
\end{proof}

\begin{remark}
\label{remark next}
It is very difficult to extend the results from the current paper to a proof of Conjecture~\ref{char p DML} beyond curves $V$ contained in $\bG_m^N$. On one hand, if $V\subset \bG_m^N$ is a higher dimensional subvariety, then one deals with products of several $F$-orbits (see Definition~\ref{definition F-sets})   and then instead of equation \eqref{the condition 7}, we would have equations of the form:
\begin{equation}
\label{the condition 13}
w_{n}=b + \sum_{i=1}^\ell a_ip^{k_im_i},
\end{equation}
where $\{w_n\}_{n\in\N_0}$ is a linear recurrence sequence, $b$ and each $a_i$ are rational numbers, while the $k_i$ are nonnegative integers. 
Finding all $n\in\N_0$ such that the linear recurrence sequence $w_{n}$ satisfies equation \eqref{the condition 13} for some $m_i\in\N_0$ is a much more difficult question. For example, a very special case of \eqref{the condition 13} would be finding all $n\in\N_0$ for which there exist some $m_i\in\N_0$ such that
$$n^d=1 + \sum_{i=1}^\ell 2^{m_i},$$
which is, yet, unsolved when $\ell>3$ and $d\ge 2$ (see \cite{cifre}).  On the other hand, if we try to solve Conjecture~\ref{char p DML} in the case of curves $V$ contained in some abelian variety $X$ defined over a finite field $\Fq$, going through a similar argument as in the proof of Theorem~\ref{main result} we would need to solve equally difficult polynomial-exponential equations of the form 
\begin{equation}
\label{the condition 14}
Q(n)=v_m,
\end{equation}
where $Q(x)\in\Q[x]$ and $\{v_m\}$ is a linear recurrence sequence whose characteristic roots are the eigenvalues corresponding the the Frobenius $F$ acting on $X$. Since these eigenvalues, i.e., the roots of the minimal equation with integer coefficients satisfied by $F$ have all absolute value equal to $q^{\frac{1}{2}}$ (according to the Weil conjectures for abelian varieties over finite fields), the equation \eqref{the condition 14} is particularly difficult. 

Finally,  we note that it is unknown whether Moosa-Scanlon result from \cite{Moosa-Scanlon} (see also \cite{groups}) extends to non-isotrivial semiabelian varieties $X$; hence, in the non-isotrivial case of semiabelian varieties, Conjecture~\ref{char p DML} is even more difficult since we do not have  a complete description of the intersection of a subvariety of $X$ with a finitely generated subgroup of $X$ as in Theorem~\ref{Moosa-Scanlon theorem}. 
\end{remark}


\end{document}